\documentclass[english]{amsart}

\usepackage{esint}
\usepackage[svgnames]{xcolor} 
\usepackage{color}
\usepackage[colorlinks,citecolor=red,pagebackref,hypertexnames=false,breaklinks]{hyperref}
\usepackage{pgf,tikz}
\usepackage{pdfsync}

\usepackage{dsfont}
\usepackage{url}
\usepackage[utf8]{inputenc}
\usepackage[T1]{fontenc}
\usepackage{lmodern}
\usepackage{babel}
\usepackage{mathtools}  
\usepackage{amssymb}
\usepackage{lipsum}
\usepackage{mathrsfs}
\usepackage{color}
\usepackage[skip=2.2pt plus 1pt, indent=12pt]{parskip}
\usepackage{stmaryrd}
\usepackage{soul}

\usepackage{accents}
\newlength{\dhatheight}
\newcommand{\doublehat}[1]{%
    \settoheight{\dhatheight}{\ensuremath{\hat{#1}}}%
    \addtolength{\dhatheight}{-0.35ex}%
    \hat{\vphantom{\rule{1pt}{\dhatheight}}%
    \smash{\hat{#1}}}}

\newtheorem{theorem}{Theorem}[section]
\newtheorem{proposition}{Proposition}[section]
\newtheorem{lemma}{Lemma}[section]

\newtheorem{corollary}{Corollary}[section]

\newtheorem{remark}{Remark}[section]

\numberwithin{equation}{section}

\title[Spectral Barron spaces of periodic functions]{Spectral Barron spaces of periodic functions}

\thanks{This work was supported by National Key Research and Development Programs of China (No. 2023YFA1009103), JSPS KAKENHI Grant Numbers JP25K17280, JP23KK0049, NSFC (No. 92570106) and Science and Technology Commission of Shanghai Municipality (No. 23JC1400501).}

\author{Mourad Choulli}
\address{Universit\'{e} de Lorraine, 34 cours L\'{e}opold, 54052 Nancy cedex, France}
\email{mourad.choulli@univ-lorraine.fr}

\author{Shuai Lu}
\address{School of Mathematics Sciences, SKLCAM and LMNS, Fudan University, 220 Handan Road, Shanghai 200433, China}
\email{slu@fudan.edu.cn}

\author{Hiroshi Takase}
\address{Department of Mathematics, Okayama University, 3-1-1 Tsushima-naka, Kita-ku, Okayama 700-8530, Japan}
\email{takase@math.okayama-u.ac.jp}

\date{}
 
\begin{document}
\begin{abstract}
Spectral Barron spaces are Fourier–Lebesgue-type spaces that arise in the approximation of functions by neural networks. Conventional spectral Barron spaces are mostly considered on the whole space or on a bounded domain with certain smoothness assumptions on the boundary. In this work, we investigate spectral Barron spaces for periodic functions, thereby extending the study of spectral Barron spaces to high-dimensional tori and, more generally, to locally compact Abelian groups. These extensions are carried out by considering Fourier series expansions that are isometrically isomorphic to weighted sequence spaces built upon the \(\ell^1\) sequence space. We discuss the well-posedness of PDEs in these newly introduced spaces. Two specific inverse problems for evolution equations are investigated, yielding conditional stability within the framework of spectral Barron spaces on the torus.
\end{abstract}

\subjclass[2020]{46B26,46B45,46N20,34A33,34A34,35A01,35A02,35R30.}

\keywords{Spectral Barron spaces, $n$-dimensional torus, locally compact Abelian group, sectorial operator, differential equations, inverse problems.}

\maketitle

\tableofcontents

\section{Introduction}\label{s1}

Spectral Barron spaces are Fourier–Lebesgue-type function spaces that have emerged as a central framework in the mathematical analysis of neural network approximation. The origin of this theory can be traced to Barron's seminal work \cite{Ba}, which established that functions whose Fourier transform is summable against a suitable weight can be universally approximated by a shallow neural network, i.e.,
\begin{align}\label{eq_nn}
f_m = \frac{1}{m}\sum_{j=1}^m a_j \sigma(w_j \cdot x + b_j),
\end{align}
where $\sigma$ is the activation function, \(m\) is the number of neurons, \( (a_j, w_j)\) denotes the outer- and inner-layer weights, respectively, and \(b_j\) denotes the inner-layer bias. This universal approximation achieves an \(L^2\)-error of order \(O(m^{-\frac{1}{2}})\), with a constant depending only on the Barron norm of the function and not on the input dimension. The dimension-independent rate, relying on Maurey's empirical method and its adaptation to Banach spaces of Rademacher type $2$, is the defining feature that makes spectral Barron spaces particularly suitable for high-dimensional problems, where classical approximation schemes suffer from the curse of dimensionality.

In recent years, the theory of spectral Barron spaces has undergone substantial development in several directions. One line of research has focused on sharpening the approximation-theoretic properties of spectral Barron spaces. \cite{KB_IEEE} extends the universal approximation property and rates for neural networks towards the ReLU type activation function. \cite{LiaoMing} proved uniform approximation rates for functions in the spectral Barron space, demonstrating that a deep neural network can achieve certain approximation rate, thereby reducing the smoothness requirement, and showing that increasing network depth can improve the approximation order for functions with low smoothness. \cite{SX23} characterized spectral Barron spaces as variation spaces corresponding to shallow neural networks and established sharp embeddings between spectral Barron spaces and Besov spaces with embedding constants independent of the input dimension. \cite{AC26} developed a quantitative approximation theory for shallow neural networks within a time-frequency analysis framework, further enriching the approximation-theoretic toolbox for Barron-type spaces.

Another significant line of inquiry concerns the connection between spectral Barron spaces and the regularity theory of partial differential equations. \cite{Xu20} studied the universal approximation rate of neural networks for elliptic equations. \cite{CLL21} investigated the representation of solutions to elliptic PDEs in Barron spaces, while \cite{CLLZ} established a regularity theory for static Schrödinger equations on \(\mathbb{R}^n\) in the spectral Barron space setting. In a related direction, the solution theory of Hamilton–Jacobi–Bellman equations in spectral Barron spaces has been developed, showing that, under suitable assumptions on the coefficients, the solution can be approximated by shallow neural networks without the curse of dimensionality \cite{FL26}. In our recent work as shown in \cite{CLT,CLT2}, we have systematically developed the functional-analytic foundations and the theories of elliptic and heat equations in various spectral Barron spaces. Moreover, attention has also turned to inverse problems within the spectral Barron space framework. For instance, the interpolation theory of spectral Barron spaces has been developed and applied to inverse problems through a link condition that yields conditional stability estimates for operator-type inverse equations \cite{LuM}. In a parallel development, \cite{CLT2} derived a logarithmic conditional stability estimate for the backward heat equation, which is the first stability estimate for inverse problems within the spectral Barron space setting. Together, these developments demonstrate that spectral Barron spaces provide a natural functional-analytic framework not only for approximation theory but also for the analysis of high-dimensional PDEs and their inverse problems.

Despite these above advances, the overwhelming majority of the literature on spectral Barron spaces has been confined to the Euclidean setting, either the whole space \(\mathbb{R}^n\) or a bounded domain \(\Omega \subset \mathbb{R}^n\) with suitable boundary smoothness. However, many problems of practical and theoretical interest, particularly in the analysis of periodic phenomena, lattice models, and problems posed on compact groups, call for a function space framework tailored to periodic functions. The natural geometric setting for such problems is the high-dimensional torus \(\mathbb{T}^n = (\mathbb{R}/2\pi\mathbb{Z})^n\), for instance as in \cite{LLW21} or, more generally, a locally compact Abelian group \(G\). Nevertheless, a systematic development of spectral Barron spaces for periodic functions, together with applications to PDEs and their inverse problems, remains largely unexplored.

The present work addresses this gap by investigating spectral Barron spaces for periodic functions. Our approach is based on the Fourier series expansion of periodic functions, which is isometrically isomorphic to a weighted sequence space built upon the \(\ell^1\) sequence space. This isomorphism provides a natural and computationally tractable characterization of the spectral Barron norm in the periodic setting, and it enables us to transfer various functional-analytic properties from the sequence space setting to the function space setting. The extension carries the investigation of spectral Barron spaces beyond the Euclidean domain to high-dimensional tori and, more generally, to locally compact Abelian groups. 
Within this framework, we investigate the well-posedness properties of PDEs and two specific inverse problems. These results demonstrate that the spectral Barron space framework on the torus provides a viable setting for the analysis of PDEs and inverse problems, complementing the existing conditional stability theory developed in the Euclidean setting c.f. \cite{LuM,CLT}.

The remainder of this paper is organized as follows. In Section \ref{s2}, we define the spectral Barron spaces of one-dimensional periodic functions and provide their properties, including duality, convolution, and embedding. Section \ref{s3} systematically investigates the well-posedness of differential equations on the proposed spectral Barron spaces, where fractional, nonlinear, and variable-coefficient linear PDEs are revisited. In Section \ref{s4}, we extend the above results to locally compact Abelian groups and the high-dimensional torus. The final section, Section \ref{s5}, discusses two inverse problems for evolution equations, which broadens our framework to include conditional stability of inverse problems for PDEs.

%
%
%
%
%

\section{Spectral Barron spaces of 1-D periodic functions and their properties}\label{s2}
In this section, we formally define the spectral Barron spaces of periodic functions and present their properties along the lines of those for the whole space in \cite{CLT}.

\subsection{Definition and duality}\label{sb2.1}

Let \(C_{\mathrm{per}}^0\) denote the space of continuous \(2\pi\)-periodic functions \(f:\mathbb R\to\mathbb C\), endowed with the norm
\[
\|f\|_{C_{\mathrm{per}}^0} := \|f\|_{C^0([0,2\pi])}.
\]

We define \(S\) to be the space of complex sequences indexed by \(\mathbb Z\). Recall that \(\ell^1\subset S\) consists of those sequences \(a=(a_j)_{j\in\mathbb Z}\) such that
\[
\|a\|_{\ell^1} := \sum_{j\in\mathbb Z}|a_j|<\infty,
\]
and that \((\ell^1,\|\cdot\|_{\ell^1})\) is a Banach space. For simplicity, we henceforth write \(a=(a_j)\) in place of \(a=(a_j)_{j\in\mathbb Z}\).

For \(f\in C_{\mathrm{per}}^0\), its Fourier coefficients are given by
\[
\hat f_j := \frac{1}{2\pi}\int_0^{2\pi} f(x)e^{-ijx}\,dx,\qquad j\in\mathbb Z,
\]
and we set \(\hat f:=(\hat f_j)\).

We define the spectral Barron space \(B_{\mathrm{per}}^0\) by
\[
B_{\mathrm{per}}^0 := \bigl\{ f\in C_{\mathrm{per}}^0 : \hat f\in \ell^1 \bigr\},
\]
equipped with the norm
\[
\|f\|_{B_{\mathrm{per}}^0} := \|\hat f\|_{\ell^1},\qquad f\in B_{\mathrm{per}}^0.
\]

Let $f\in B_{\mathrm{per}}^0$. Then the sequence $(g_k)$ given by
\[
g_k(x):=\sum_{j=-k}^k\hat{f}_je^{ijx},\quad x\in \mathbb{R},\; k\in \mathbb{N},
\]
converges uniformly on $\mathbb{R}$, and therefore its limit $g:=\sum_{j\in \mathbb{Z}}\hat{f}_je^{ijx}$ belongs to $C_{\mathrm{per}}^0$. Since $\hat{g}_j=\hat{f}_j$ for all $j\in \mathbb{Z}$, the uniqueness theorem (e.g. \cite[Theorem 2.7]{Ka}) yields $g=f$. In other words, we have
\begin{equation}\label{ff}
f(x)=\sum_{j\in \mathbb{Z}}\hat{f}_je^{ijx},\quad x\in \mathbb{R}^n,\; f\in B_{\mathrm{per}}^0.
\end{equation}
On the other hand, if $a=(a_j)\in \ell ^1$, then the function $f$ defined by
\[
f(x)=\sum_{j\in \mathbb{Z}}a_je^{ijx},\quad x\in \mathbb{R}^n,
\]
satisfies $\hat{f}=a$, and therefore $f\in B_{\mathrm{per}}^0$. Again by the uniqueness theorem (e.g. \cite[Theorem 2.7]{Ka}), we obtain that the map $f\in B_{\mathrm{per}}^0\mapsto \hat{f}\in \ell^1$ is bijective. In summary, $B_{\mathrm{per}}^0$ is isometrically isomorphic to $\ell ^1$. More precisely,
\[
I_0: B_{\mathrm{per}}^0\rightarrow \ell^1: f\mapsto \hat{f}
\]
is an isometric isomorphism. Consequently, the adjoint of $I_0$,
\[
I_0^\ast : \ell^\infty \rightarrow [B_{\mathrm{per}}^0]^\ast,
\]
is also an isometric isomorphism, where $\ell^\infty$ is the Banach space of complex bounded sequences, equipped with the sup norm, and is usually identified with $(\ell^1)^\ast$. Recall that the duality pairing between $\ell^1$ and $\ell^\infty$ is given by
\[
\langle b,a\rangle_{\infty,1}:=\sum_{j\in \mathbb{Z}}a_j\overline{b}_j,\quad a=(a_j)\in \ell^1,\; b=(b_j)\in \ell^\infty.
\]
If the duality pairing between $B_{\mathrm{per}}^0$ and $[B_{\mathrm{per}}^0]^\ast$ is denoted by $\langle \cdot, \cdot \rangle_0$, then
\[
\langle g,f\rangle_0=\langle (I_0^{-1})^\ast(g),I_0(f)\rangle_{\infty,1},\quad f\in B_{\mathrm{per}}^0,\; g\in [B_{\mathrm{per}}^0]^\ast.
\]

We proceed further by defining spectral Barron spaces of periodic functions of arbitrary nonnegative order. First, let $\langle j\rangle := \sqrt{1+j^2}$, $j\in \mathbb{Z}$, and for $s\ge 0$, set
\[
\ell^{1,s}:=\{a=(a_j)\in S;\; (\langle j\rangle^s \hat{f}_j)\in \ell^1\}.
\]
$\ell^{1,s}$ is a Banach space when endowed with the norm
\[
\|a\|_{\ell^{1,s}}:=\|(\langle j\rangle^s a_j)\|_{\ell^1},\quad a=(a_j)\in \ell^{1,s}.
\]
 
The following formula will be useful later:
\begin{equation}\label{df}
\widehat{f^{(k)}}_j=(ij)^k \hat{f}_j,\quad k\ge 0 ,
\end{equation}
provided that the left-hand side of the identity above is well defined.

We define the spectral Barron spaces of periodic functions of order $s\ge 0$ as follows:
\[
B_{\mathrm{per}}^s:=\{ f\in C_{\mathrm{per}}^0;\; \hat{f}\in \ell^{1,s}\}\subset C^{\lfloor s\rfloor}\; (\mbox{by \eqref{df}}).
\]
We verify that $B_{\mathrm{per}}^s$ is complete with respect to the norm
\[
\|f\|_{B_{\mathrm{per}}^s}:=\|(\hat{f}_j)\|_{\ell^{1,s}}.
\]
As above, the mapping
\[
I_s: B_{\mathrm{per}}^s\rightarrow \ell^{1,s}:f\mapsto \hat{f}
\]
is an isometric isomorphism.

Using the following formal observation
\[
\sum_{j\in \mathbb{Z}}a_j\overline{b}_j=\sum_{j\in \mathbb{Z}}[\langle j\rangle^sa_j][\langle j\rangle^{-s}\overline{b}_j],
\]
we show that $[\ell^{1,s}]^\ast$ can be identified with
\[
\ell^{\infty,-s}:=\{ a=(a_j)\in S;\; (\langle j\rangle^{-s}a_j)\in \ell^\infty\},
\]
equipped with the norm
\[
\| (a_j)\|_{\ell^{\infty,-s}}:=\|(\langle j\rangle^{-s}a_j)\|_{\ell^\infty}.
\]
Furthermore, $[B_{\mathrm{per}}^s]^\ast$ can be identified with $\ell^{\infty,-s}$, and if $\langle \cdot ,\cdot\rangle_s$ denotes the duality pairing between $B_{\mathrm{per}}^s$ and $[B_{\mathrm{per}}^s]^\ast$, then we have
\[
\langle g,f\rangle_s=\langle (I_s^\ast)^{-1}(g),I_s(f)\rangle_{\infty,1},\quad f\in B_{\mathrm{per}}^s,\; g\in [B_{\mathrm{per}}^s]^\ast.
\]

Also, note that for all $0\le s<t$, $B_{\mathrm{per}}^t\hookrightarrow B_{\mathrm{per}}^s$ and
\[
\|f\|_{B_{\mathrm{per}}^s}\le \|f\|_{B_{\mathrm{per}}^t}, \quad f\in B_{\mathrm{per}}^t.
\]

\begin{remark}\label{rem1}
{\rm
Let $\mathbb{T}:=\{z\in \mathbb{C};\; |z|=1\}$ be the one-dimensional torus. We have $\mathbb{T}\cong \mathbb{R}/\mathbb{Z}$ via identifying $x\in \mathbb{R}/\mathbb{Z}$ with $t=e^{2i\pi x}\in \mathbb{T}$. As pointed out in \cite[Section 2.4]{Fo}, the groups $\mathbb{T}$ and $\mathbb{Z}$ are dual to each other; the natural dual measures on them are the normalized Lebesgue measure $\frac{dt}{2\pi}$ on $\mathbb{T}$ and the counting measure on $\mathbb{Z}$. The Fourier transform on $\mathbb{T}$, viewed as a compact Abelian group, is given by
\[
\hat{f}(j)=\frac{1}{2\pi}\int_0^{2\pi}f(t)e^{-ijt}dt,\quad j\in \mathbb{Z}
\]
and its inverse is
\[
f(t)=\sum_{j\in \mathbb{Z}}\hat{f}(j)e^{ijt},\quad t\in \mathbb{T}.
\]
For $s\ge 0$, we define the Barron space $B^s(\mathbb{T})$ as follows:
\[
B^s(\mathbb{T}):=\{f\in C(\mathbb{T});\; (\hat{f}(j))\in \ell^{1,s}\}.
\]
With the norm
\[
\|f\|_{B^s(\mathbb{T})}:=\|(\hat{f}(j))\|_{\ell^{1,s}}, \quad f\in B^s(\mathbb{T}),
\]
$B^s(\mathbb{T})$ is a Banach space that can be identified with $B_{\mathrm{per}}^s$.

In the higher-dimensional setting, we will discuss, in particular, spectral Barron spaces on locally compact Abelian groups in Section~\ref{s4}.
}
\end{remark}

\subsection{Product, convolution and interpolation}\label{sb2.2}
We define the convolution of $a=(a_j)\in \ell^1$ and $b=(b_j)\in \ell^1$ as follows:
\begin{equation}\label{cf}
a\ast b:=c,\quad c=(c_j),\; c_j:=\sum_{k\in \mathbb{Z}}a_kb_{j-k}.
\end{equation}
Thus one can easily verify that $a\ast b\in \ell^1$ and
\begin{equation}\label{cn}
\|a\ast b\|_{\ell^1}\le \|a\|_{\ell^1}\|b\|_{\ell^1}.
\end{equation}
Moreover, $a\ast b=b\ast a$, and if $d\in \ell ^1$, then there holds $a\ast(b\ast d)=(a\ast b)\ast d$.

Let $f,g\in B_{\mathrm{per}}^0$. Because of the absolute and uniform convergence of the series used below, and in view of \eqref{ff}, we have
\begin{align*}
\widehat{(fg)}_j&=\frac{1}{2\pi}\int_0^{2\pi}f(x)g(x)e^{-ijx}dx
\\
&=\sum_{k\in \mathbb{Z}}\hat{f}_k \left(\frac{1}{2\pi}\int_0^{2\pi}g(x)e^{-i(j-k)x}dx\right)
\\
&=\sum_{k\in \mathbb{Z}}\hat{f}_k\hat{g}_{j-k}.
\end{align*}
In other words, by \eqref{cf}, we have $\widehat{fg}=\hat{f}\ast \hat{g}$. Thus, $fg\in B_{\mathrm{per}}^0$, and from \eqref{cn} we derive
\begin{equation}\label{pr1}
\|fg\|_{B_{\mathrm{per}}^0}\le \|f\|_{B_{\mathrm{per}}^0}\|g\|_{B_{\mathrm{per}}^0}.
\end{equation}

More generally, for all $s\ge 0$, using Peetre's inequality
\[
\langle j\rangle^s\le 2^{\frac{s}{2}}\langle k\rangle^s\langle j-k\rangle^s,\quad j,k\in \mathbb{Z},
\]
we can verify that the convolution is well defined on $\ell^{1,s}$ and
\begin{equation}\label{cn2}
\|a\ast b\|_{\ell^{1,s}}\le2^{\frac{s}{2}} \|a\|_{\ell^{1,s}}\|b\|_{\ell^{1,s}}.
\end{equation}
As above, we transfer this result to $B_{\mathrm{per}}^s$. Specifically, if $f,g\in B_{\mathrm{per}}^s$, then $fg\in B_{\mathrm{per}}^s$ and
\begin{equation}\label{pr2}
\|fg\|_{B_{\mathrm{per}}^s}\le 2^{\frac{s}{2}}\|f\|_{B_{\mathrm{per}}^s}\|g\|_{B_{\mathrm{per}}^s}.
\end{equation}

Let
\[
L_{\mathrm{per}}^1:=\{f:\mathbb{R}\rightarrow \mathbb{C}\; \mbox{measurable};\; f\; \mbox{is $2\pi$-periodic and}\; f\in L^1(]0,2\pi[)\},
\]
which we endow with the norm
\[
\|f\|_{L_{\mathrm{per}}^1}:=\frac{1}{2\pi}\|f\|_{L^1(]0,2\pi[)}.
\]
We recall that for all $f,g\in L_{\mathrm{per}}^1$
\[
(f\ast g)(x):=\frac{1}{2\pi}\int_0^{2\pi}f(x-y)g(y)dy,
\]
and we have
\begin{equation}\label{pr2.1}
\widehat{(f\ast g)}_j=\hat{f}_j\hat{g}_j,\quad j\in \mathbb{Z}.
\end{equation}
We verify that if $g\in L_{\mathrm{per}}^1$, then $\hat{g}\in \ell^\infty$ and $\|\hat{g}\|_{\ell^\infty}\le \|g\|_{L_{\mathrm{per}}^1}$. Consequently, using \eqref{pr2.1}, we obtain that for all $f\in B_{\mathrm{per}}^s$ and $g\in L_{\mathrm{per}}^1$, $f\ast g=g\ast f\in B_{\mathrm{per}}^s$ and
\[
\|f\ast g\|_{B_{\mathrm{per}}^s}\le \|f\|_{B_{\mathrm{per}}^s}\|g\|_{L_{\mathrm{per}}^1}.
\]

We also note from \eqref{pr2.1} that if $f\in B_{\mathrm{per}}^s$ and $g\in B_{\mathrm{per}}^0$, then $f\ast g\in B_{\mathrm{per}}^s$ and
\[
\|f\ast g\|_{B_{\mathrm{per}}^s}\le \|f\|_{B_{\mathrm{per}}^s}\|g\|_{B_{\mathrm{per}}^0}.
\]
This is an immediate consequence of the fact that
\[
\|\hat{g}\|_{\ell^\infty}\le \|\hat{g}\|_{\ell^1}=\|g\|_{B_{\mathrm{per}}^0}.
\]

At the end of this subsection, we discuss interpolation between different spectral Barron spaces. Let $0\le r<s<t$ and $f\in B_{\mathrm{per}}^t$. Since $s=(1-\alpha) r+\alpha t$ with $\alpha=\frac{s-r}{t-r}$, using the decomposition
\[
\sum_{j\in \mathbb{Z}}\langle j\rangle^s|\hat{f}_j|=\sum_{j\in \mathbb{Z}}[\langle j\rangle^{(1-\alpha) r}|\hat{f}_j|^{1-\alpha}][\langle j\rangle^{\alpha t}|\hat{f}_j|^\alpha]
\]
together with Hölder's inequality, we obtain the interpolation inequality 
\begin{equation}\label{ii}
\|f\|_{B_{\mathrm{per}}^s}\le \|f\|_{B_{\mathrm{per}}^r}^{1-\alpha}\|f\|_{B_{\mathrm{per}}^t}^\alpha.
\end{equation}

\subsection{Embedding property}\label{sb2.3}
First, we recall some classic definitions. Let $\ell^2$ be the space of sequences $a=(a_j)\in S$ such that $\sum_j|a_j|^2<\infty$. By equipping with the norm
\[
\|a\|_{\ell^2}=\left(\sum_{j\in \mathbb{Z}}|a_j|^2\right)^{\frac{1}{2}},\quad a=(a_j)\in \ell^2,
\]
we know that $\ell^2$ is a Hilbert space. For later use, we note that $\ell^1\hookrightarrow \ell^2$ with
\begin{equation}\label{1,2}
\|a\|_{\ell ^2}\le \|a\|_{\ell^1},\quad a\in \ell ^1.
\end{equation}
Also, define
\[
L_{\mathrm{per}}^2:=\{f:\mathbb{R}\rightarrow \mathbb{C}\; \mbox{measurable};\; f\; \mbox{is $2\pi$-periodic and}\; f\in L^2(]0,2\pi[\},
\]
which we endow with the norm
\[
\|f\|_{L_{\mathrm{per}}^2}:=\frac{1}{\sqrt{2\pi}}\|f\|_{L^2(]0,2\pi[)}.
\]
We recall the following Bessel–Parseval identity:
\[
\|f\|_{L_{\mathrm{per}}^2}=\|\hat{f}\|_{\ell^2}, \quad f\in L_{\mathrm{per}}^2.
\]

For $s\ge 0$, consider the Sobolev space
\[
H_{\mathrm{per}}^s=\{f\in L_{\mathrm{per}}^2;\; (\langle j\rangle^s\hat{f}_j)\in L_{\mathrm{per}}^2\}.
\]
We verify that $H_{\mathrm{per}}^s$ is a Hilbert space when endowed with the norm
\[
\|f\|_{H_{\mathrm{per}}^s}=\|(\langle j\rangle^s\hat{f}_j)\|_{L_{\mathrm{per}}^2},\quad f\in H_{\mathrm{per}}^s.
\]
For convenience, we set $H_{\mathrm{per}}^0:=L_{\mathrm{per}}^2$.

Let $f\in B^s_{\mathrm{per}}$. In light of \eqref{1,2}, we have
\[
\|f\|_{H_{\mathrm{per}}^s}=\|(\langle j\rangle^s \hat{f}_j)\|_{\ell^2} \le \|(\langle j\rangle^s \hat{f}_j)\|_{\ell^1}=\|f\|_{B^s_{\mathrm{per}}}.
\]
That is, we have proved the following result.
 
\begin{proposition}\label{pro0}
For all $s\ge 0$, $B^s_{\mathrm{per}}\hookrightarrow H_{\mathrm{per}}^s$ with
\[
\|f\|_{H_{\mathrm{per}}^s}\le \|f\|_{B^s_{\mathrm{per}}},\quad f\in B^s_{\mathrm{per}}.
\]
\end{proposition}

We now establish the following compact embedding result.

\begin{proposition}\label{pro1}
For all $0\le s<t$, the embedding $B_{\mathrm{per}}^t\hookrightarrow B_{\mathrm{per}}^s$ is compact.
\end{proposition}

\begin{proof}
It suffices to establish that the embedding $\ell^{1,t}\hookrightarrow \ell^{1,s}$ is compact. Let $(a^n)_{n\in \mathbb{N}}$ be a sequence in the unit ball of $\ell^{1,t}$, with $a^n=(a_j^n)$. Then, we have
\begin{equation}\label{com1}
\sum_{j\in \mathbb{Z}}\langle j\rangle^t|a_j^n|\le 1,\quad n\in \mathbb{N}.
\end{equation}
Let $\epsilon >0$. As $\lim_{|j|\rightarrow \infty}\langle j\rangle^{s-t}=0$, there exists $j_0\in \mathbb{N}$ such that $\langle j\rangle^{s-t}\le \frac{\epsilon}{4}$ for all $|j|> j_0$. Thus,
\begin{align*}
\sum_{|j|> j_0}\langle j\rangle^s|a_j^n-a_j^m|&=\sum_{|j|> j_0}\langle j\rangle^{s-t}[\langle j\rangle^t|a_j^n-a_j^m|]
\\
&\le \frac{\epsilon}{4} \sum_{|j|> j_0}\langle j\rangle^t|a_j^n-a_j^m|,\quad n,m\in \mathbb{N},
\end{align*}
which, combined with \eqref{com1}, yields
\begin{equation}\label{com2}
\sum_{|j|> j_0}\langle j\rangle^s|a_j^n-a_j^m|\le \frac{\epsilon}{2},\quad n,m\in \mathbb{N}.
\end{equation}
On the other hand, \eqref{com1} shows that the sequence $((a_j^n)_{|j|\le j_0})_n$ is bounded in $\mathbb{C}^{2j_0+1}$, and hence admits a convergent subsequence $((a_j^{n_k})_{|j|\le j_0})_k$. In particular, $((a_j^{n_k})_{|j|\le j_0})_k$ is a Cauchy sequence, and hence there exists $k_0\in \mathbb{N}$ such that
\begin{equation}\label{com3}
\sum_{|j|\le j_0}\langle j\rangle^s|a_j^{n_{k_1}}-a_j^{n_{k_2}}|\le \frac{\epsilon}{2}, \quad k_1\ge k_0,\; k_2\ge k_0.
\end{equation}
Putting together \eqref{com2} and \eqref{com3}, we get
\[
\sum_{j\in \mathbb{Z}}\langle j\rangle^s|a_j^{n_{k_1}}-a_j^{n_{k_2}}|\le \epsilon, \quad k_1\ge k_0,\; k_2\ge k_0.
\]
That is, $(a^{n_k})_k$ is a Cauchy sequence in $\ell^{1,s}$ and therefore converges in $\ell^{1,s}$. The proof is complete. 
\end{proof}

The space $B_{\mathrm{per}}^s$ is ``close'' to an interpolation space, as the following theorem shows.
\begin{theorem}\label{thm1}
For $0\le r<s<t$, $\theta=\frac{s-r}{t-r}$ and $0<\tilde{\theta}<\theta$, there holds
\[
(B^r_{\mathrm{per}},B^t_{\mathrm{per}})_{\theta,1}\hookrightarrow B_{\mathrm{per}}^s\hookrightarrow (B^r_{\mathrm{per}},B^t_{\mathrm{per}})_{\tilde{\theta},1}.
\]
\end{theorem}

\begin{proof}
The proof is a modification of that of \cite[Theorem 2.2]{CLT}. The first embedding is a consequence of \eqref{ii} and \cite[Proposition 1.20]{Lu}. To prove the second embedding, for $f\in B_{\mathrm{per}}^s$ and $\mu >0$, we write $f=g_\mu +h_\mu$, where
\[
g_\mu (x):=\sum_{|j|>\mu}\hat{f}_je^{ijx}\quad\text{and}\quad h_\mu(x):=\sum_{|j|\le \mu}\hat{f}_je^{ijx}.
\] 
We verify that $g_\mu\in B_{\mathrm{per}}^r$, $h_\mu \in B_{\mathrm{per}}^t$ and
\[
\|g_\mu\|_{B_{\mathrm{per}}^r}\le \mu^{-(s-r)}\|f\|_{B_{\mathrm{per}}^s}\quad\text{and}\quad \|h_\mu\|_{B_{\mathrm{per}}^t}\le \mu^{t-s}\|f\|_{B_{\mathrm{per}}^s}.
\]
The remainder of the proof is analogous to that of \cite[Theorem 2.2]{CLT}, and we therefore omit it.
\end{proof}


\begin{theorem}\label{thm2}
Let $k\in \mathbb{N}\cup\{0\}$ and $0<\theta<1$. Then $B_{\mathrm{per}}^{k+\theta}\hookrightarrow C_b^{k, \theta}$.
\end{theorem}

\begin{proof}
It suffices to provide the proof for $k=0$, which consists of an adaptation of the proof in \cite[Proposition 2.6]{CLSSS}. Let $f\in B_{\mathrm{per}}^\theta$. For $x,y\in \mathbb{R}^n$, using that
\[
|e^{ijx}-e^{ijy}|=|e^{ij(x-y)}-1|\le 2^{1-\theta}{|j(x-y)|^\theta}|,
\]
we obtain
\begin{align*}
|g(x)-g(y)|&=\left| \sum_{j\in \mathbb{Z}}(e^{ijx}-e^{ijy})\hat{f}_j\right|
\\
&\le 2^{1-\theta}{|(x-y)|^\theta}|  \sum_{j\in \mathbb{Z}}|j|^\theta|\hat{f}_j|
\\
&\le 2^{1-\theta}{|(x-y)|^\theta}|\|f\|_{B^\theta}.
\end{align*}
This yields the expect embedding.
\end{proof}

We proceed as for \cite[Corollary 2.1]{CLT} to obtain the following consequence of Theorem \ref{thm2}, where
\[
\dot{B}^s:=\{f_{|[0,2\pi]};\; f\in B_{\mathrm{per}}^s\}
\]
is equipped with the norm of $B_{\mathrm{per}}^s$. The following embedding then holds true. 

\begin{corollary}\label{cor1}
Let $k\in \mathbb{N}\cup\{0\}$, $0<\theta<1$ and $p\ge 1$. For all $0<t<k +\frac{\theta}{p}$, $\dot{B}^s\hookrightarrow W^{t,p}(]0,2\pi[)$.
\end{corollary}

\section{Differential equations on $B_{\mathrm{per}}^0$}\label{s3}

In this section, we discuss several differential equations on $B_{\mathrm{per}}^0$. To this end, for all $\kappa \in \mathbb{R}$, we define
\[
\mathbb{C}_\kappa:=\{z\in \mathbb{C};\; \Re z>\kappa\}.
\]

\subsection{The Schr\"odinger operator on $B_{\mathrm{per}}^0$}\label{sb3.1}
We first define the unbounded operator $A_0: B_{\mathrm{per}}^0\rightarrow B_{\mathrm{per}}^0$ by
\[
A_0u=u'',\quad u\in D(A_0)=B_{\mathrm{per}}^2.
\]
The resolvent set of $A_0$ is defined by
\[
\rho(A_0):=\{z\in \mathbb{C};\; z-A_0\colon D(A_0)\rightarrow B^0_\mathrm{per}\; \text{is a bijection and}\; (z-A)^{-1}\in \mathscr{B}(B^0_\mathrm{per})\},
\]
where $\mathscr{B}(B^0_\mathrm{per})$ denotes the space of bounded linear operators on $B^0_\mathrm{per}$.

Let $\lambda\in \mathbb{C}_0$, $f\in B_{\mathrm{per}}^0$, and consider the equation
\begin{equation}\label{Sc1}
\lambda u-u''=f.
\end{equation}
If \eqref{Sc1} admits a solution $u\in B_{\mathrm{per}}^2$, then necessarily there holds
\[
(\lambda +j^2)\hat{u}_j=\hat{f}_j,\quad j\in \mathbb{Z},
\]
and therefore $u$ is explicitly given by
\[
u(x)=\sum_{j\in \mathbb{Z}}\frac{\hat{f}_j}{\lambda+j^2}e^{ijx},\quad x\in \mathbb{R}.
\]
This solution belongs to $B_{\mathrm{per}}^2$ because $\sup_{j\in \mathbb{Z}}\frac{\langle j\rangle^2}{|\lambda+j^2|}<\infty$.  Therefore, we obtain
\[
\|u\|_{B_{\mathrm{per}}^2}\le \|f\|_{B_{\mathrm{per}}^0},\quad \lambda\in \mathbb{C}_1.
\]
In addition, since $|\lambda+j^2|\ge |\lambda|$ for all $\lambda\in \mathbb{C}_0$, we have
\[
\|u\|_{B_{\mathrm{per}}^0}\le \frac{1}{|\lambda|}\|f\|_{B_{\mathrm{per}}^0},\quad \lambda\in \mathbb{C}_0.
\]
In other words, $\mathbb{C}_0\subset \rho(A_0)$ and
\begin{equation}\label{Sc2}
\|(\lambda-A_0)^{-1}\|_{\mathscr{B}(B_{\mathrm{per}}^0)}\le \frac{1}{|\lambda|},\quad \lambda\in \mathbb{C}_0.
\end{equation}

Let $q\in B_{\mathrm{per}}^0$. We have seen before that $B_{\mathrm{per}}^0$ is an algebra for the product of functions. Therefore, we can define the unbounded operator $A_q: B_{\mathrm{per}}^0\rightarrow B_{\mathrm{per}}^0$ by
\[
A_qu=u''+qu,\quad u\in D(A_q)=B_{\mathrm{per}}^2.
\]
Let $\lambda\in \mathbb{C}_0$, $f\in B_{\mathrm{per}}^0$, and $T_q(\lambda): B_{\mathrm{per}}^0\rightarrow B_{\mathrm{per}}^0:u\mapsto (\lambda-A_0)^{-1}(qu)$. Then consider the equation
\begin{equation}\label{Sc3}
u-T_q(\lambda)u=(\lambda-A_0)^{-1}(f).
\end{equation}
If $\lambda_q:=\|q\|_{B_{\mathrm{per}}^0}$, then we have $\|T_q(\lambda)\|_{\mathscr{B}(B_{\mathrm{per}}^0)}\le \frac{\lambda_q}{|\lambda|}$. In particular, $\|T_q(\lambda)\|_{\mathscr{B}(B_{\mathrm{per}}^0)}<1$ for all $\lambda\in \mathbb{C}_{\lambda_q}$. Consequently, \eqref{Sc3} admits a unique solution provided that $\lambda\in \mathbb{C}_{\lambda_q}$. In addition, $u=T_q(\lambda)u+(\lambda-A_0)^{-1}(f)\in B_{\mathrm{per}}^2$ and
\[
\|u\|_{B_{\mathrm{per}}^0}=\|(1-T_q(\lambda))^{-1}(\lambda-A_0)^{-1}(f)\|_{B_{\mathrm{per}}^0}\le \frac{1}{|\lambda|-\lambda_q}\|f\|_{B_{\mathrm{per}}^0}.
\]
In particular, for all $\lambda\in \mathbb{C}_{2\lambda_q}$ it holds that
\begin{equation}\label{Sc4}
\|u\|_{B_{\mathrm{per}}^0}=\|(1-T_q(\lambda))^{-1}(\lambda-A_0)^{-1}(f)\|_{B_{\mathrm{per}}^0}\le \frac{2}{|\lambda|}\|f\|_{B_{\mathrm{per}}^0}.
\end{equation}
As a solution of \eqref{Sc3} is also a solution of the equation
\[
(\lambda-A_q) u=f,
\]
we obtain that $\rho(A_q)\supset \mathbb{C}_{2\lambda_q}$, and from \eqref{Sc4} we have
\[
\|(\lambda -A_q)^{-1}\|_{\mathscr{B}(B_{\mathrm{per}}^0)}\le \frac{2}{|\lambda|},\quad \lambda \in 
\mathbb{C}_{2\lambda_q} .
\]
Hence, according to \cite[Proposition 2.1.11]{Lun}, $A_q$ is a sectorial operator, and hence it is the generator of an analytic semigroup $(e^{tA_q})_{t\ge 0}$. Since $A_q$ has a dense domain in $B^0_\mathrm{per}$, this semigroup is strongly continuous.

\begin{remark}\label{rem2}
{\rm
With a slight modification, the analysis we carried out for the Schr\"odinger operator $A_q$ can be adapted to the following nonlocal operator:
\[
S_K u=u''+K\ast u,\quad u\in D(S_K):=B_{\mathrm{per}}^2,
\]
where $K\in L^1_{\mathrm{per}}$. In this case, $\lambda_q$ and $T_q$ have to be replaced respectively by $\lambda_K:=\|K\|_{L^1_{\mathrm{per}}}$ and
\[
T_K:B_{\mathrm{per}}^0\rightarrow B_{\mathrm{per}}^0:u\mapsto T_K(u):=(\lambda -A_0)^{-1}(K\ast u).
\]
As for the Schr\"odinger operator $A_q$, we verify that $\rho(S_K)\supset \mathbb{C}_{2\lambda_K}$, and we have
\[
\|(\lambda-S_K)^{-1}\|_{\mathscr{B}(B_{\mathrm{per}}^0)}\le \frac{2}{|\lambda|},\quad \lambda \in \mathbb{C}_{2\lambda_K}.
\]
Therefore, again according to \cite[Proposition 2.1.11]{Lun}, $S_K$ is a sectorial operator.
}
\end{remark}

\subsection{Fractional powers of $-A_0$} \label{sb3.2}

There is a natural way to define $(-A_0)^{\alpha}$ for $0<\alpha \le 1$. The formula
\[
-A_0(u)(x)=\sum_{j\in \mathbb{Z}}j^2\hat{u}_je^{ijx},\quad x\in \mathbb{R}^n,\; u\in B_{\mathrm{per}}^2,
\]
suggests defining $(-A_0)^\alpha: B_{\mathrm{per}}^0\rightarrow B_{\mathrm{per}}^0$ as the unbounded operator given by
\[
(-A_0)^{\alpha}(u)(x):=\sum_{j\in \mathbb{Z}}j^{2\alpha}\hat{u}_je^{ijx},\quad x\in \mathbb{R}^n,\; u\in D((-A_0)^{\alpha}) :=B_{\mathrm{per}}^{2\alpha}.
\]

For $\lambda\in \mathbb{C}_0$ and $f\in B_{\mathrm{per}}^0$, we consider the equation
\begin{equation}\label{Sc5}
\lambda u + (-A_0)^{\alpha}u=f.
\end{equation}
As above, if $u\in B_{\mathrm{per}}^{2\alpha}$ is a solution of the above equation, then necessarily there holds
\[
(\lambda +j^{2\alpha})\hat{u}_j=\hat{f}_j,\quad j\in \mathbb{Z}.
\]
Consequently, \eqref{Sc5} admits a unique solution $u\in B_{\mathrm{per}}^{2\alpha}$ given by
\[
u(x)=\sum_{j\in \mathbb{Z}}\frac{1}{\lambda +j^{2\alpha}}\hat{f}_je^{ijx},\quad x\in \mathbb{R},
\]
from which we obtain
\[
\|u\|_{B_{\mathrm{per}}^0}\le \frac{1}{|\lambda|}\|f\|_{B_{\mathrm{per}}^0}.
\]
In other words, $\rho(-(-A_0)^{\alpha})\supset \mathbb{C}_0$ and
\[
\|(\lambda+(-A_0)^{\alpha})^{-1}\|_{\mathscr{B}(B_{\mathrm{per}}^0)}\le \frac{1}{|\lambda|},\quad \lambda\in \mathbb{C}_0.
\]
Again, we deduce from \cite[Proposition 2.1.11]{Lun} that $-(-A_0)^{\alpha}$ is a sectorial operator. Note that $-(-A_0)^1=A_0$.

\subsection{A nonlinear differential equation}\label{sb3.3}

In this subsection, we adapt the arguments used in Subsection \ref{sb3.1} to study a nonlinear differential equation. Let $k\in \mathbb{N}\setminus\{1\}$ and $q, f\in B_{\mathrm{per}}^0$ with $q\ne 0$. Consider the following nonlinear equation
\begin{equation}\label{spe1}
\lambda u-u''=qu^k+f.
\end{equation}
For $\lambda \in \mathbb{C}_0\subset \rho(A_0)$, if $u\in B_{\mathrm{per}}^2$ is a solution of \eqref{spe1}, then $u$ is a fixed point of the mapping
\[
N: B_{\mathrm{per}}^0\rightarrow B_{\mathrm{per}}^0:w\mapsto N(w)=(\lambda-A_0)^{-1}(qw^k+f).
\]
Let $w_1,w_2\in B_{\mathrm{per}}^0$. From the results of Subsection \ref{sb3.1}, we obtain
\begin{equation}\label{spe2}
\|N(w_1)-N(w_2)\|_{B_{\mathrm{per}}^0}\le \frac{\|q\|_{B_{\mathrm{per}}^0}}{|\lambda|}\|w_1^k-w_2^k\|_{B_{\mathrm{per}}^0}.
\end{equation}
Using the formula
\[
w_1^k-w_2^k=(w_1-w_2)\sum_{m=0}^{k-1}w_1^m w_2^{k-1-m},
\]
we get
\[
\|w_1^k-w_2^k\|_{B_{\mathrm{per}}^0}\le \|w_1-w_2\|_{B_{\mathrm{per}}^0}\sum_{m=0}^{k-1}\|w_1\|_{B_{\mathrm{per}}^0}^m\|w_2\|^{k-1-m}_{B_{\mathrm{per}}^0}.
\]
Using this inequality in \eqref{spe2} yields
\[
\|N(w_1)-N(w_2)\|_{B_{\mathrm{per}}^0}\le \frac{\|q\|_{B_{\mathrm{per}}^0}}{|\lambda|}\|w_1-w_2\|_{B_{\mathrm{per}}^0}\sum_{m=0}^{k-1}\|w_1\|_{B_{\mathrm{per}}^0}^m\|w_2\|^{k-1-m}_{B_{\mathrm{per}}^0}.
\]

Let $\mathbf{B}_\rho$ be the ball centered at the origin with radius $\rho>0$. Then we have
\begin{align*}
&\|N(w)\|_{B_{\mathrm{per}}^0}\le \frac{1}{|\lambda|}\left(\|q\|_{B_{\mathrm{per}}^0}\rho^k+\|f\|_{B_{\mathrm{per}}^0}\right),\quad w\in \mathbf{B}_\rho,
\\
&\|N(w_1)-N(w_2)\|_{B_{\mathrm{per}}^0}\le \frac{\|q\|_{B_{\mathrm{per}}^0}k\rho^{k-1}}{|\lambda|}\|w_1-w_2\|_{B_{\mathrm{per}}^0},\quad w_1,w_2\in \mathbf{B}_\rho.
\end{align*}
Consequently, under the following conditions
\begin{align}
&\|q\|_{B_{\mathrm{per}}^0}\rho^k+\|f\|_{B_{\mathrm{per}}^0}\le \rho |\lambda|,\label{spe3}
\\
&\|q\|_{B_{\mathrm{per}}^0}k\rho^{k-1}<|\lambda|,\label{spe4}
\end{align}
$N:\mathbf{B}_\rho\rightarrow \mathbf{B}_\rho$ is a contraction. Therefore, it admits a unique fixed point $u\in \mathbf{B}_\rho\cap B_{\mathrm{per}}^2$. This fixed point is the unique solution of \eqref{spe1} belonging to $\mathbf{B}_\rho\cap B_{\mathrm{per}}^2$.

For instance, if we choose
\begin{equation}\label{spe5}
0<\rho<\left(\frac{|\lambda|}{2k\|q\|_{B_{\mathrm{per}}^0}}\right)^{\frac{1}{k-1}}
\quad\text{and}\quad
\|f\|_{B_{\mathrm{per}}^0}\le\left(1-\frac{1}{2k}\right)\rho |\lambda|,
\end{equation}
then one verifies that \eqref{spe3} and \eqref{spe4} hold. In particular, conditions \eqref{spe5} are satisfied whenever $\rho$, $q$, and $f$ are given and $|\lambda|$ is chosen sufficiently large. Of course, other choices are possible.

\begin{theorem}\label{thmnl}
Let $\rho>0$, $\lambda \in \mathbb{C}_0$, and $q, f\in B_{\mathrm{per}}^0$ be such that \eqref{spe3} and \eqref{spe4} hold. Then equation \eqref{spe1} admits a unique solution $u\in \mathbf{B}_\rho\cap B_{\mathrm{per}}^2$.
\end{theorem}

\subsection{Linear differential equation with variable coefficients}\label{sb3.4}
For $\lambda \in \mathbb{C}_0$, $a\in B_{\mathrm{per}}^1$, and $f\in B_{\mathrm{per}}^0$, we consider the differential equation
\begin{equation}\label{vc1}
\lambda u-(au')'=f.
\end{equation}
Let us first explain why it is difficult to solve this equation using the method employed for \eqref{Sc1}. If we proceed in a manner analogous to that used for \eqref{Sc1}, we obtain the following infinite system of equations:
\[
\lambda\hat{u}_j+\sum_{k\in \mathbb{Z}}(k^2+k(j-k))\hat{u}_k\hat{a}_{j-k}=\hat{f}_j,\quad j\in \mathbb{Z}.
\]
It therefore appears quite difficult to solve this infinite system of equations. To circumvent this difficulty, as in the case of the nonlinear equation treated in the previous subsection, we will use a perturbation argument. Let $a_0>0$ be fixed. We rewrite \eqref{vc1} in the following form
\[
\lambda u-a_0 u''= ((a-a_0)u')'+f.
\]
Equivalently, one can derive
\[
\frac{\lambda}{a_0} u-u''= \left(\left(\frac{a}{a_0}-1\right)u'\right)'+\frac{f}{a_0}.
\]
Replacing $\frac{\lambda}{a_0}$, $\frac{a}{a_0}$, and $\frac{f}{a_0}$ by $\lambda$, $a$, and $f$, we may assume, without loss of generality, that $a_0=1$. Therefore, the following equation then holds true
\begin{equation}\label{vc2}
\lambda u-u''= ((a-1)u')'+f.
\end{equation}
We reformulate this equation as
\begin{equation}\label{vc3}
u-T(u)= (\lambda-A_0)^{-1}f,
\end{equation}
where
\[
T: B_{\mathrm{per}}^2\rightarrow  B_{\mathrm{per}}^2: u\mapsto T(u):=(\lambda-A_0)^{-1}[((a-1)u')'].
\]
We verify that $u\in B_{\mathrm{per}}^2$ is a solution of \eqref{vc2} if and only if it is a solution of \eqref{vc3}. Since
\[
\|(\lambda-A_0)^{-1}\|_{\mathscr{B}(B_{\mathrm{per}}^0,B_{\mathrm{per}}^2)}\le 1, \quad \lambda \in \mathbb{C}_1,
\]
we obtain from \eqref{pr2}
\begin{align*}
\|Tu\|_{B_{\mathrm{per}}^2}&\le \|((a-1)u')'\|_{B_{\mathrm{per}}^0}
\\
&\le 2 \|a-1\|_{B_{\mathrm{per}}^1}\|u\|_{B_{\mathrm{per}}^2}.
\end{align*}
Consequently, under the assumptions $\lambda \in \mathbb{C}_1$ and $2\|a-1\|_{B_{\mathrm{per}}^1}\le \delta$ for some $0<\delta<1$ , the equation \eqref{vc3} admits a unique solution $u=(1-T)^{-1}(\lambda-A_0)^{-1}f$ and
\[
\|u\|_{B_{\mathrm{per}}^2}\le \frac{1}{1-\delta}\|f\|_{B_{\mathrm{per}}^0}.
\]

Let $c=(c_j)\in \ell^{1,1}$ be such that 
$\|c\|_{\ell^{1,1}}\le \delta/2$ for some $0<\delta<1$. Then the function $a$ defined by
\[
a(x)=1+\sum_{j\in \mathbb{Z}}c_je^{ijx},\quad x\in \mathbb{R},
\]
belongs to $B_{\mathrm{per}}^1$ and satisfies $2\|a-1\|_{B_{\mathrm{per}}^1}\le \delta$.

We are now in a position to state the following result.

\begin{theorem}\label{thmvc}
Let $0<\delta <1$. For all $\lambda\in \mathbb{C}_1$, $f\in B_{\mathrm{per}}^0$, and $a\in B_{\mathrm{per}}^1$ such that $2\|a-1\|_{B_{\mathrm{per}}^1}\le \delta$, equation \eqref{vc1} has a unique solution $u\in B_{\mathrm{per}}^2$ satisfying
\[
\|u\|_{B_{\mathrm{per}}^2}\le \frac{1}{1-\delta}\|f\|_{B_{\mathrm{per}}^0}.
\]
\end{theorem}

\section{Spectral Barron spaces on locally compact Abelian groups and high dimensional extension}\label{s4}

In this section, we discuss the extension of the theory of spectral Barron spaces on $\mathbb{R}^n$, which we studied in \cite{CLT}, and on $\mathbb{T}$, which we addressed in Section \ref{s2}. Since the setting is very general, the results obtained are less rich than those established for $\mathbb{R}^n$ and $\mathbb{T}$. In the case of $\mathbb{T}^n$, the $n$-dimensional torus, we show that the results coincide with those established for $\mathbb{T}$.

For the reader's convenience, we use the same notation as in Folland's monograph \cite{Fo}.

\subsection{Preliminaries on locally compact Abelian groups}\label{sb4.1}

For simplicity and without any restriction, we consider multiplicative groups, and we recall that a topological group is called locally compact group if its topology is Hausdorff and locally compact. 
In the remainder part of this section, unless otherwise stated, $G$ is  a locally compact Abelian group. The identity element of $G$ will be denoted as $e$.

A continuous homomorphism $\xi$ of $G$ into the unit circle $\mathbb{T}$ is called (unitary) character of $G$. We denote by $\hat{G}$ the set of characters of $G$, and we write
\[
\langle x,\xi\rangle:=\xi (x),\quad x\in G,\; \xi \in \hat{G}.
\]
For further use, we recall that there holds
\[
\langle xy,\xi\rangle=\langle x,\xi\rangle\langle y,\xi\rangle,\;
\langle e,\xi\rangle =1,\; \text{and}\; \langle x^{-1},\xi\rangle=\overline{\langle x,\xi\rangle}.
\]
Notice that $\hat{G}$ is usually endowed with the topology of compact convergence on $G$, under which the group operations are continuous. This topology coincides with the weak-$\ast$ topology on $L^\infty (G)$, inherits as a subset of $L^\infty(G)$. Equipped with the above mentioned topology, $\hat{G}$ is a locally compact Abelian group, we call it the dual group of $G$. From \cite[Theorem 4.6 and Corollary 4.8]{Fo}, $\widehat{\mathbb{R}^n}\cong \mathbb{R}^n$ with pairing $\langle x,\xi\rangle=e^{2i\pi x\cdot \xi}$, $x,\xi\in \mathbb{R}^n$, and $\hat{\mathbb{T}}\cong \mathbb{Z}$, with pairing $\langle t,j\rangle=t^j$, $t\in \mathbb{T}$, $j\in \mathbb{Z}$. We further note that that it is customary to identify $\doublehat{G}$ with $G$.

We recall that a Haar measure on $G$ is a nonzero nonnegative Radon measure $\mu$ on $G$ such that for all $x\in G$ and a Borel set $E$ of $G$ we have $\mu (xE) (=\mu (Ex)) =\mu (E)$. From \cite[Theorem 2.10]{Fo}, $G$ admits a Haar measure $\mu$ and, by \cite[Theorem 2.20]{Fo}, if $\mu'$ is another Haar measure on $G$, then $\mu'=c\mu$, for some constant $c>0$.

In the following, $\mu$ will denote a Haar measure on $G$. For $1\le p\le\infty$,  we set $L^p(G):=L^p(G,d\mu(x))$.
For $f\in L^1(G)$, its Fourier transform is defined as follows
\[
\mathscr{F}f(\xi)=\hat{f}(\xi)=\int_Gf(x)\overline{\langle x,\xi\rangle}d\mu(x),\quad \xi\in \hat{G}.
\]
Before giving a Fourier inversion formula, we need to introduce new definitions. A function $\phi\in L^\infty(G)$ is said to be of a positive type if
\[
\int_G\int_G\phi(y^{-1}x)f(x)\overline{f(y)}d\mu(x)d\mu(y)\ge 0,\quad f\in L^1(G).
\]
For $X\in \{G,\hat{G}\}$, let 
\begin{align*}
&C(X):=\{ f:X\rightarrow \mathbb{C};\; f\; \mbox{is continuous}\},
\\
&C_{\mathrm{b}}(X)= C(X)\cap L^\infty(X) \; \mbox{(equipped with the sup-norm)},
\\
&C_{\mathrm{c}}(X):=\{f\in C(X),\; \mathrm{supp}(f)\; \mbox{is compact}\},
\\
&C_0(X)=[\mbox{closure of}\; C_{\mathrm{c}}(X)\; \mbox{in}\; C_{\mathrm{b}}(X)]=\{f\in C(X);\; f\; \mbox{vanishes at}\; \infty\}
\end{align*}
and set
\[
\mathcal{P}(G):=\{\phi \in C(G);\; \phi\; \mbox{is of positive type}\}.
\]
In the following, the space of bounded complex Radon measures on $\hat{G}$ will be denoted by $M(\hat{G})$, that is the topological dual of $C_0(\hat{G})$. By \cite[Proposition 4.18]{Fo}, the mapping
\[
\lambda \in M(\hat{G})\mapsto \phi_\lambda \in C_b(G);\; \phi_\lambda(x):=\int_{\hat{G}}\langle \xi,x\rangle d\lambda(\xi)
\]
is a linear injection, and
\[
\|\phi_\lambda\|_{C_b(G)}\le \|\lambda\|_{M(\hat{G})}:=|\lambda|(\hat{G}).
\]
According to Bochner's theorem (e.g. \cite[Theorem 4.19]{Fo}), if $\phi\in \mathcal{P}(G)$, then there exists a unique positive $\lambda \in M(\hat{G})$ such that $\phi=\phi_\lambda$.

Let
\begin{align*}
&\mathcal{B}(G):=\{\phi_\lambda;\; \lambda\in M(\hat{G})\}\subset C_b(G),
\\
&\mathcal{B}^p(G):=\mathcal{B}(G)\cap L^p(G),\quad 1\le p<\infty.
\end{align*}
Note that $\mathcal{B}(G)$ is the linear span of $\mathcal{P}(G)$. 
By the Fourier inversion theorem in \cite[Theorem 4.22]{Fo}, if $f\in \mathcal{B}^1(G)$, then $\hat{f}\in L^1(\hat{G})$. Furthermore, there exists a Haar measure $\hat{\mu}$ on $\hat{G}$ such that we have for all $f\in \mathcal{B}^1(G)$ 
\begin{equation}\label{lcg1}
f(x)=\overline{\mathscr{F}}(\hat{f})(x):=\int_{\hat{G}}\hat{f}(\xi)\langle x,\xi\rangle d\hat{\mu}(\xi).
\end{equation}
We call $\overline{\mathscr{F}}$ the Fourier cotransform.

It should be noted that, according to \cite[Theorem 4.33]{Fo}, the formula \eqref{lcg1} holds for all $f\in L^1(G)$ such that $\hat{f}\in L^1(\hat{G})$, and by \cite[Lemma 4.30]{Fo}, $\mathscr{F}(\mathcal{B}^1(G))$ is dense in $L^1(\hat{G})$.

\subsection{Definitions}\label{sb4.2}

We are now able to define the spectral Barron space $B^0(G)$ by
\[
B^0(G):=\{f\in L^1(G);\; \hat{f}\in L^1(\hat{G})\}.
\]
Similarly, we define 
\[
B^0(\hat{G}):=\{f\in L^1(\hat{G});\; \overline{\mathscr{F}}(f)\in L^1(G)\}.
\]
We will concentrate our study to $B^0(G)$.

As noted in the remarks following \cite[Proposition 10.17]{GQ}, the restriction of $\mathscr{F}$ to $B^0(G)$ is a bijection onto $B^0(\hat{G})$; its inverse is the restriction of $\overline{\mathscr{F}}$ to $B^0(\hat{G})$.

Equipped with the norm
\[
\|f\|_{B^0(G)}:=\|\hat{f}\|_{L^1(\hat{G})},
\]
$B^0(G)$ is a Banach space, and by \eqref{lcg1} we have $B^0(G)\hookrightarrow C_{\mathrm{b}}(G)$ with
\[
\|f\|_{C_{\mathrm{b}}(G)}\le \|f\|_{B^0(G)} .
\]

Let  $\Sigma :=C(\hat{G},[1,\infty))$. For $\sigma \in \Sigma$ and $s\ge 0$, define 
\[
B_\sigma ^s(G):=\{f\in L^1(G);\; \sigma^s\hat{f}\in L^1(\hat{G})\}.
\]
We verify that $B_\sigma ^s(G)$ endowed with norm
\[
\|f\|_{B_\sigma ^s(G)}:=\|\sigma^s\hat{f}\|_{L^1(\hat{G})}
\]
is a Banach space. Furthermore, we have obviously
\begin{align*}
&B_\sigma ^0(G)=B^0(G),\quad \sigma \in \Sigma,
\\
&B_{\sigma_1}^{s_1}\hookrightarrow B_{\sigma_2}^{s_2},\quad 0\le s_2\le s_1,\; \sigma_1,\sigma_2\in \Sigma ;\; \sigma_2\le \sigma_1.
\end{align*}

It should be noted that the spaces $B_\sigma^s(G)$ differ slightly from the spectral Barron spaces $B^s$ introduced in \cite{CLT} when $G=\mathbb{R}^n$ and $\sigma(\xi)=\sqrt{1+|\xi|^2}$ ($\xi \in \mathbb{R}^n$), but coincide with the spaces $B^s(\mathbb{T})$ when $G=\mathbb{T}$ and $\sigma(j)=\sqrt{1+j^2}$ ($j \in \mathbb{Z}$).

Let $\sigma \in \Sigma$ and $0\le r<s<t$. Then $s=(1-\alpha) r+\alpha t$ with $\alpha=\frac{s-r}{t-r}$. For $f\in B_\sigma^t(G)$, in light of the identity
\[
\sigma^s|\hat{f}|=[\sigma^{(1-\alpha) r}|\hat{f}|^{1-\alpha}][\sigma^{\alpha t}|\hat{f}|^\alpha],
\]
applying H\"older's inequality, we obtain
\[
\|f\|_{B_\sigma ^s(G)}\le \|f\|_{B_\sigma ^r(G)}^{1-\alpha}\|f\|_{B_\sigma ^t(G)}^\alpha.
\]

\subsection{Convolution}\label{sb4.3}

The convolution of $f,g\in L^1(G)$ is defined as follows:
\[
(f\ast g)(x)=\int_Gf(y)g(y^{-1}x)d\mu(y).
\]
Note that as $G$ is Abelian, we have $f\ast g=g\ast f$.

Replacing $G$ and $\mu$ respectively by $\hat{G}$ and $\hat{\mu}$, we get the definition of the convolution on $L^1(\hat{G})$. In particular, we have for $f,g\in B_\sigma^s(G)$
\[
(\hat{f}\ast \hat{g})(\xi )=\int_{\hat{G}}\hat{f}(\eta)\hat{g}(\eta^{-1}\xi)d\hat{\mu}(\eta),
\]
and then
\begin{equation}\label{lcg5}
\sigma(\xi)^{s}|(\hat{f}\ast \hat{g})(\xi )|\le \int_{\hat{G}}\sigma(\xi)^{s}|\hat{f}(\eta)\hat{g}(\eta^{-1}\xi)|d\hat{\mu}(\eta).
\end{equation}
We assume that $\sigma$ satisfies in addition the following Peetre type inequality, where $\mathbf{c}_s>0$ is a constant,
\[
\sigma(\xi)^{s}\le \mathbf{c}_s\sigma(\eta)^s\sigma(\eta^{-1}\xi)^s,\quad \eta,\; \xi \in \hat{G}.
\]
Under this new assmption, \eqref{lcg5} implies
\begin{align}
\int_{\hat{G}}\sigma(\xi)^{s}|(\hat{f}\ast \hat{g})(\xi )|&d\hat{\mu}(\xi)\le \int_{\hat{G}}\int_{\hat{G}}\sigma(\xi)^{s}|\hat{f}(\eta)||\hat{g}(\eta^{-1}\xi)|d\hat{\mu}(\eta)d\hat{\mu}(\xi)\label{lcg6}
\\
&\le \mathbf{c}_s\int_{\hat{G}}\sigma(\eta)^{s}|\hat{f}(\eta)|d\hat{\mu}(\eta)\int_{\hat{G}}\sigma(\eta^{-1}\xi)^s|\hat{g}(\eta^{-1}\xi)|d\hat{\mu}(\xi)\nonumber
\\
&\le \mathbf{c}_s\int_{\hat{G}}\sigma(\eta)^{s}|\hat{f}(\eta)|d\hat{\mu}(\eta)\int_{\hat{G}}\sigma(\xi)^s|\hat{g}(\xi)|d\hat{\mu}(\xi).\nonumber
\end{align}
On the other hand, applying Fubini's theorem, we obtain
\begin{align*}
\widehat{fg}(\xi)&=\int_Gf(x)g(x)\overline{\langle x,\xi\rangle}d\mu(x)
\\
&=\int_Gf(x)\int_{\hat{G}}\hat{g}(\eta)\langle x,\eta \rangle \overline{\langle x,\xi\rangle} d\hat{\mu}(\eta)d\mu(x)
\\
&=\int_{\hat{G}}\hat{g}(\eta)\left(\int_Gf(x)\overline{\langle x,\eta^{-1}\xi\rangle}d\mu(x)\right)d\hat{\mu}(\eta),
\end{align*}
where we used that $\eta^{-1}(x)=\overline{\eta (x)}$ (e.g. \cite[Proposition 10.3]{GQ}). Hence,
\[
\widehat{fg}(\xi)=\int_{\hat{G}}\hat{g}(\eta)\hat{f}(\eta^{-1}\xi)d\hat{\mu}(\eta)=(\hat{f}\ast \hat{g})(\xi).
\]
This, combined with \eqref{lcg6}, gives $fg\in B_\sigma^s$ and
\[
\|fg\|_{B_\sigma^s}\le \mathbf{c}_s\|g\|_{B_\sigma^s}\|g\|_{B_\sigma^s}.
\]

Let $f\in B_\sigma ^s$  and $g\in L^1(G)$. Using the formula
\[
\langle x,\xi\rangle=\langle y,\xi\rangle\langle y^{-1}x,\xi\rangle,\quad x,y\in G,\; \xi \in \hat{G},
\]
we obtain 
\begin{align*}
\widehat{f\ast g}(\xi)&=\int_G \int_G \overline{\langle x,\xi\rangle}f(y)g(y^{-1}x)d\mu(y)d\mu(x)
\\
&=\int_G \overline{\langle y,\xi\rangle}f(y)\left(\int_G \overline{\langle y^{-1}x,\xi\rangle}g(y^{-1}x)d\mu(x)\right)d\mu(y)
\\
&=\int_G \overline{\langle y,\xi\rangle}f(y)\left(\int_G \overline{\langle x,\xi\rangle}g(x)d\mu(x)\right)d\mu(y)
\\
&=\hat{f}(\xi)\hat{g}(\xi).
\end{align*}
This shows that $f\ast g\in B_\sigma^s(G)$ and
\[
\|f\ast g\|_{B_\sigma^s(G)}\le \|f\|_{B_\sigma^s(G)}\|g\|_{L^1(G)}.
\]

Let $A(G)$ be the subspace of $L^1(G)$ spanned by $f\ast g$ with $f,g\in L^2(G)$. Again, as it is remarked in the comments after \cite[Proposition 10.17]{GQ}, $A(G)$ is an ideal of $L^1(G)$. Furthermore, if $f\in A(G)$, then $\hat{f}\in L^1(G)$ and the inversion formula \eqref{lcg1} holds.

\subsection{Duality}\label{sb4.4}

Let $\sigma \in \Sigma$ and $s\ge 0$. Denote the closure of $L^2(G)\cap B_\sigma^s(G)$ in $B_\sigma^s(G)$  by $K_\sigma^s(G)$, and define
\[
\tilde{B}_\sigma^{-s}(G)=\{g=\overline{\mathscr{F}}(\sigma^s h);\; h\in L^\infty(\hat{G}),\; \sigma^s h\in L^2(\hat{G})\} (\subset L^2(G)),
\]
that we endow with the norm
\[
\|g\|_{\tilde{B}_\sigma^{-s}(G)}:=\|\sigma^{-s}\hat{g}\|_{ L^\infty(\hat{G})}.
\]
Then consider the bilinear form:
\[
\mathfrak{b}: (L^2(G)\cap B_\sigma^s(G))\times \tilde{B}_\sigma^{-s}(G):(f,g)\mapsto \mathfrak{b}(f,g):=\int_G f(x)\overline{g(x)}d\mu(x).
\]
From the Plancherel theorem (e.g. \cite[Theorem 4.26]{Fo}) and we have
\begin{equation}\label{lcg2}
\int_G f(x)\overline{g(x)}d\mu(x)=\int_{\hat{G}}\hat{f}(\xi)\overline{\hat{g}(\xi)}d\hat{\mu}(x),\quad f,g\in L^2(G).
\end{equation}
For all $f\in L^2(G)\cap B_\sigma^s(G)$ and $g\in \tilde{B}_\sigma^{-s}(G)$, we have
\begin{align}
\left|  \int_{\hat{G}}\hat{f}(\xi)\overline{\hat{g}(\xi)}d\hat{\mu}(x)\right|&=\left|  \int_{\hat{G}}[\sigma^{s}\hat{f}(\xi)][\sigma ^{-s}\overline{\hat{g}(\xi)}]d\hat{\mu}(x)\right| \label{lcg3}
\\
&\le \|\sigma^s\hat{f}\|_{L^1(\hat{G})}\|\sigma^{-s}\hat{g}\|_{ L^\infty(\hat{G})}=\|f\|_{B_\sigma^s(G)}\|g\|_{\tilde{B}_\sigma^{-s}(G)}.\nonumber
\end{align}
\eqref{lcg2} and \eqref{lcg3} show that $\mathfrak{b}$ admits a continuous extension to $K_\sigma^s(G)\times \tilde{B}_\sigma^{-s}(G)$, denoted again by $\mathfrak{b}$, with 
\begin{equation}\label{lcg3.1}
|\mathfrak{b}(f,g)|\le \|f\|_{B_\sigma^s(G)}\|g\|_{\tilde{B}_\sigma^{-s}(G)},\quad (f,g)\in K_\sigma^s(G))\times \tilde{B}_\sigma^{-s}(G).
\end{equation}
For $g\in \tilde{B}_\sigma^{-s}(G)$, let $\mathfrak{b}_g: f\in K_\sigma^s(G)\mapsto \mathfrak{b}(f,g)$. By \eqref{lcg3.1}, $\mathfrak{b}_g\in [K_\sigma^s(G)]^\ast$ with 
\[
\|\mathfrak{b}_g\|_{[K_\sigma^s(G)]^\ast}\le \|g\|_{\tilde{B}_\sigma^{-s}(G)}.
\]
Under the additional assumption that $K_\sigma^s(G)$ is dense in $L^2(G)$, we verify that the linear continuous map
\[
g\in \tilde{B}_\sigma^{-s}(G)\mapsto \mathfrak{b}_g\in [K_\sigma^s(G)]^\ast
\]
is injective. Hence, $\tilde{B}_\sigma^{-s}(G)\hookrightarrow [K_\sigma^s(G)]^\ast$.

In the case where $G=\mathbb{R}^n$ and $\sigma(\xi)=\sqrt{1+|\xi|^2}$ ($\xi \in \mathbb{R}^n$) or $G=\mathbb{T}$ and $\sigma(j)=\sqrt{1+j^2}$ ($j \in \mathbb{Z}$), we have both $K_\sigma^s(G)=B_\sigma^s(G)$ and $B_\sigma^s(G)$ is dense in $L^2(G)$.

\subsection{Spectral Barron spaces on $\mathbb{T}^n$}\label{sb4.5}

In this subsection, we extend the results of Sections \ref{s2} and \ref{s3} to an $n$-dimensional torus. As the analysis is analogous to the one we carried out for $n=1$, we do not provide all the details.

We will use the notations:
\begin{align*}
&x\cdot y=\sum_{j=1}^n x_jy_j,\quad x=(x_1,\ldots ,x_n)\in \mathbb{R}^n,\; y=(y_1,\ldots ,y_n)\in \mathbb{R}^n,
\\
&|x|=\sqrt{x\cdot x},\quad x=(x_1,\ldots ,x_n)\in \mathbb{R}^n.
\end{align*}

The $n$-dimensional torus is the cartesian product of $n$ copies of $\mathbb{T}$: $\mathbb{T}^n:=\prod_{j=1}^n\mathbb{T}$. By \cite[Corollary 4.8]{Fo}, $\widehat{\mathbb{T}^n}\cong \mathbb{Z}^n$, from which we derive the Fourier inversion formula in $\mathbb{T}^n$:
\[
f(x)=\sum_{j\in \mathbb{Z}^n}e^{ix\cdot j}\hat{f}_j,\quad f\in B^0(\mathbb{T}^n),
\]
where
\[
\hat{f}_j=\frac{1}{(2\pi)^n}\int_{[0,2\pi]^n}f(x)e^{-i x\cdot j}dx,\quad j\in \mathbb{Z}^n.
\]

For $s> 0$, the spectral Barron space defined in Subsection \ref{sb4.2} when $\sigma (j)=\sqrt{1+|j|^2}:=\langle j\rangle$ ($j\in \mathbb{Z}^n$) will be denoted by $B^s(\mathbb{T}^n)$. Let $S^n$ be the vector space of complex-valued sequences $(a_j)_{j\in \mathbb{Z}^n}$ and
\begin{align*}
&\ell^{1,s}(\mathbb{Z}^n):=\left\{ a=(a_j)_{j\in \mathbb{Z}^n}\in S^n;\; \|a\|_{\ell^{1,s}(\mathbb{Z}^n)}:=\sum_{j\in \mathbb{Z}^n}\langle j\rangle^s|a_j|<\infty\right\},
\\
&\ell^{\infty,-s}(\mathbb{Z}^n):=\left\{ b=(b_j)_{j\in \mathbb{Z}^n}\in S^n;\; \|b\|_{\ell^{\infty,-s}(\mathbb{Z}^n)}:=\sup_{j\in \mathbb{Z}^n}\langle j\rangle^{-s}|b_j|<\infty\right\}.
\end{align*}
We verify that $(\ell^{1,s}(\mathbb{Z}^n),\|\cdot \|_{\ell^{1,s}(\mathbb{Z}^n)})$ and $(\ell^{\infty,-s}(\mathbb{Z}^n),\|\cdot\|_{\ell^{\infty,-s}(\mathbb{Z}^n)})$ are Banach spaces, and $\ell^{\infty,-s}(\mathbb{Z}^n)$ is identified with $[\ell^{1,s}(\mathbb{Z}^n)]^\ast$ with the duality pairing
\[
\langle b, a\rangle_{\infty,1}:=\sum_{j\in \mathbb{Z}^n}a_j\overline{b}_j,\quad a\in \ell^{1,s}(\mathbb{Z}^n),\; b\in \ell^{\infty,-s}(\mathbb{Z}^n).
\]
For simplicity, we use $(a_j)$ instead of $(a_j)_{j\in \mathbb{Z}^n}$. With the notation $\hat{f}=(\hat{f}_j)$, we have
\[
B^s(\mathbb{T}^n)=\{ f\in L^1(\mathbb{T}^n);\; \hat{f}\in \ell^{1,s}(\mathbb{Z}^n)\},
\]
and
\[
\|f\|_{B^s(\mathbb{T}^n)}=\|\hat{f}\|_{\ell^{1,s}(\mathbb{Z}^n)}.
\]
There exists an isometric isomorphism $\mathfrak{I}_s$ from $B^s(\mathbb{T}^n)$ onto $\ell^{1,s}(\mathbb{Z}^n)$. Then $[B^s(\mathbb{T}^n)]^\ast$ can be identified with $\ell^{\infty,-s}(\mathbb{Z}^n)$ and the duality pairing between $B^s(\mathbb{T}^n)$ and $[B^s(\mathbb{T}^n)]^\ast$, denoted by $\langle \cdot,\cdot \rangle_s$, is given as follows:
\[
\langle g,f\rangle_s:=\langle (\mathfrak{I}_s^{-1})^\ast(g), \mathfrak{I}_s(f)\rangle_{\infty,1},\quad f\in B^s(\mathbb{T}^n),\; g\in [B^s(\mathbb{T}^n)]^\ast.
\]
The spaces $B^s(\mathbb{T}^n)$ possess properties analogous to those of $B^s(\mathbb{T})$ ; we therefore omit the detailed statements.

We recall that for all $\alpha=(\alpha_1,\ldots ,\alpha_n)\in (\mathbb{N}\cup\{0\})^n$, $\partial^\alpha=\partial_1^{\alpha_1}\ldots \partial_n^{\alpha_n}$, $|\alpha|=\alpha_1+\ldots +\alpha_n$ and $x^\alpha=x_1^{\alpha_1}\ldots x_n^{\alpha_n}$. For sufficiently smooth $f:\mathbb{T}^n\rightarrow \mathbb{C}$, we have
\[
\widehat{\partial^\alpha f}_j=i^{|\alpha|}j^\alpha\hat{f}_j,\quad j\in \mathbb{Z}^n,\; \alpha\in (\mathbb{N}\cup\{0\})^n.
\]
In particular, we have
\[
\widehat{\Delta f}_j=-|j|^2 \hat{f}_j,\quad j\in \mathbb{Z}^n,
\]
and then
\[
\Delta f=-\sum_{j\in \mathbb{Z}^n}|j|^2 \hat{f}_je^{ix\cdot j},\quad f\in B^2(\mathbb{T}^n).
\]
Define the unbounded operator $A:B^0(\mathbb{T}^n)\rightarrow B^0(\mathbb{T}^n)$ by
\begin{equation}\label{Lap}
Au:=\Delta u,\quad u\in D(A):=B^2(\mathbb{T}^n).
\end{equation}
Recall that the resolvent set of $A$ is defined by
\[
\rho(A):=\{\lambda \in \mathbb{C};\; \lambda-A\colon D(A)\rightarrow B^0(\mathbb{T}^n)\; \text{is a bijection and}\; (\lambda-A)^{-1}\in \mathscr{B}(B^0(\mathbb{T}^n))\}.
\]

Let $f\in B^0(\mathbb{T}^n)$ and $\lambda \in \mathbb{C}$ such that $\Re \lambda>0$. Using that  unique solution of the equation in $B^2(\mathbb{T}^n)$ of the equation
\[
(\lambda-\Delta)u=f,
\]
is given by
\[
u=\sum_{j\in \mathbb{Z}^n}\frac{1}{\lambda +|j|^2} \hat{f}_je^{ix\cdot j},
\] 
we get that $\{\lambda \in \mathbb{C};\; \Re \lambda >0\}\subset \rho(A)$ and
\[
\|\lambda (\lambda -A)^{-1}\|_{\mathscr{B}(B^0(\mathbb{T}^n))}\le 1.
\]
Applying once again \cite[Proposition 2.1.11]{Lun}, we obtain that $A$ is sectorial. 

Note that the results established in Section \ref{s3} for $n=1$ can be adapted to any $n \ge 1$.

\section{Inverse problems in spectral Barron spaces of periodic equations}\label{s5}
In this section, we investigate two inverse problems within the framework of spectral Barron spaces for periodic equations. Both inverse problems arise from evolution equations, but their measurements are of different types.

\subsection{Backward heat equations in $B^0(\mathbb{T}^n)$}\label{sb4.6}

Let $A$ be the operator given by \eqref{Lap}. We showed in the previous subsection that $A$ is sectorial and consequently generates an analytic semigroup $(e^{tA})_{t\ge 0}$. It should be noted that most of the results established for $B^0(\mathbb{R}^n)$ in \cite{CLT2} remain valid, with minor modifications, for $B^0(\mathbb{T}^n)$. We leave the details to the interested reader.

Simple calculations show that $e^{tA}$ is given as follows:
\[
e^{tA}f=\sum_{j\in \mathbb{Z}^n}e^{-|j|^2t+ix\cdot j}\hat{f}_j,\quad f\in B^0(\mathbb{T}^n).
\]
Hence,
\begin{equation}\label{sg1}
\widehat{e^{tA}f}_j=e^{-|j|^2t }\hat{f}_j,\quad j\in \mathbb{Z}^n.
\end{equation}
In particular, we have
\[
\sup_{t\ge 0}\|e^{tA}\|_{\mathscr{B}(B^0(\mathbb{T}^n))}= 1.
\]

We fix $\tau >0$ and we consider the space
\[
\mathfrak{R}_\tau(\mathbb{T}^n):=\{ h\in C^0(\mathbb{T}^n);\; (e^{\tau|j|^2}\hat{h}_j)\in \ell^1(\mathbb{Z}^n)\},
\]
that we endow with norm
\[
\|h\|_{\mathfrak{R}_\tau(\mathbb{T}^n)}:=\|(e^{\tau|j|^2}\hat{h}_j)\|_{\ell^1(\mathbb{Z}^n)}.
\]

If $g\in \mathfrak{R}_\tau(\mathbb{T}^n)$, then $f=\sum_{j\in \mathbb{Z}} e^{|j|^2\tau +ix\cdot j}\hat{g}_j\in B^0(\mathbb{T}^n)$, and we verify  that $e^{\tau A}f=g$. This and \eqref{sg1} gives the following result.

\begin{lemma}\label{lemsg1}
The operator
\begin{equation}\label{sg2}
\mathfrak{I}_\tau:f\in B^0(\mathbb{T}^n)\mapsto e^{\tau A}f\in \mathfrak{R}_\tau(\mathbb{T}^n)
\end{equation}
is an isometric isomorphism.
\end{lemma}

Note that $\mathfrak{I}_\tau$ is a smoothing operator, since $\mathfrak{R}_\tau(\mathbb{T}^n) \subset \bigcap_{s\ge 0}B^s(\mathbb{T}^n)$.

\begin{proposition}\label{prosg1}
Let $\mathfrak{I}_\tau$ be defined by \eqref{sg2}. We have for all $s>0$ and $f\in B^s(\mathbb{T}^n)$
\begin{equation}\label{sg3}
\|f\|_{B^0(\mathbb{T}^n)}\le e^{\rho \tau}\|\mathfrak{I}_\tau(f)\|_{B^0(\mathbb{T}^n)}+\rho^{-\frac{s}{2}}\|f\|_{B^s(\mathbb{T}^n)},\quad \rho >0.
\end{equation}
\end{proposition}

\begin{proof}
Let $f\in B^0(\mathbb{T}^n)$ and $g=e^{\tau A}f$. According to formula \eqref{sg1}, we obtain
\begin{equation}\label{sg4}
\sum_{|j|\le \sqrt{\rho}}|\hat{f}_j|=\sum_{|j|\le \sqrt{\rho}}e^{|j|^2\tau }|\hat{g}_j|\le e^{\rho \tau}\|g\|_{B^0(\mathbb{T}^n)},\quad \rho >0.
\end{equation}
On the other hand, we have
\begin{equation}\label{sg5}
\sum_{|j|> \sqrt{\rho}}|\hat{f}_j|=\sum_{|j|> \sqrt{\rho}}\langle j\rangle ^{-s}[\langle j\rangle ^s|\hat{f}_j|]\le \rho^{-\frac{s}{2}}\|f\|_{B^s(\mathbb{T}^n)},\quad \rho >0.
\end{equation}
Putting together \eqref{sg4} and \eqref{sg5}, we get \eqref{sg3}.
\end{proof}

As $\|\mathfrak{I}_\tau(f)\|_{B^0(\mathbb{T}^n)}<\|f\|_{B^s(\mathbb{T}^n)}$ if $f\in B^s(\mathbb{T}^n)\setminus\{0\}$,  by minimizing the right hand side of \eqref{sg3} with respect to $\rho$, we obtain the following corollary.

\begin{corollary}\label{corsg1}
Let $\mathfrak{I}_\tau$ be defined by \eqref{sg2}. We have for all $s>0$ and $f\in B^s(\mathbb{T}^n)\setminus\{0\}$
\[
\frac{\|f\|_{B^0(\mathbb{T}^n)}}{\|f\|_{B^s(\mathbb{T}^n)}}\le \left(\tau+\frac{s}{2}\right)^{\frac{s}{2}}\left|\ln \frac{\|\mathfrak{I}_\tau(f)\|_{B^0(\mathbb{T}^n)}}{\|f\|_{B^s(\mathbb{T}^n)}}\right|^{-\frac{s}{2}}.
\]
\end{corollary}

The previous results we established for the initial-to-final state operator $\mathfrak{I}_\tau$ are exactly the same as those obtained in \cite{CLT2} when $\mathbb{T}^n$ is replaced by $\mathbb{R}^n$.

\subsection{Inverse source problems in $B^0(\mathbb{T}^n)$}

In this subsection, the notations are those of the previous one.
Let $(b_j)$ a sequence in $L^1(]0,\tau[)$ satisfying: there exists a sequence $(\epsilon_j)$ in $[0,1[$ such that $\epsilon:=\|(\epsilon_j)\|_{\ell ^\infty(\mathbb{Z}^n)}<1$, and
\[
\int_0^\tau |b_j(s)|e^{s|j|^2}ds\le \epsilon_j,\quad j\in \mathbb{Z}^n.
\]
Set
\[
a_j(t)=1+\int_0^t b_j(s)ds,\quad t\in [0,\tau],\; j\in \mathbb{Z}^n.
\]
Define 
\[
L: f\in B^0(\mathbb{T}^n)\mapsto L(f)\in B^0(\mathbb{T}^n): L(f)(t):=\sum_{j\in \mathbb{Z}}a_j(t)e^{ix\cdot j}\hat{f}_j,
\]
and 
\[
S(f)(t)=\int_0^te^{sA}L(f)(t-s)ds,\quad t\ge 0,\; f\in B^0(\mathbb{T}^n).
\]
Note that $u:=S(f)$ is the mild solution of the equation
\[
u'(t)=Au(t)+ L(f)(t),\; t>0, \quad u(0)=0.
\]

\begin{proposition}\label{proie1}
For all $k\in \mathfrak{R}_\tau(\mathbb{T}^n)$, there exists a unique $f\in B^0(\mathbb{T}^n)$ such that $S(f)'(\tau)=k$. In addition, we have
\begin{equation}\label{ie0}
\|f\|_{B^0(\mathbb{T}^n)}\le (1-\epsilon)^{-1}\|k\|_{\mathfrak{R}_\tau(\mathbb{T}^n)}.
\end{equation}
\end{proposition}

\begin{proof}
Let $k\in \mathfrak{R}_\tau(\mathbb{T}^n)$. As
\[
S(f)'(t)=e^{tA}f+\int_0^te^{sA}L(f)'(t-s)ds,\quad t\ge 0,
\]
we seek $f\in B^0(\mathbb{T}^n)$ such that
\begin{equation}\label{ie1}
e^{\tau A}f+\int_0^\tau e^{sA}L(f)'(\tau-s)ds=k.
\end{equation}
Let $g=\mathfrak{I}_\tau(f)\in \mathfrak{R}_\tau(\mathbb{T}^n)$. In this case, $\hat{f}_j=e^{\tau|j|^2}\hat{g}_j$ and then
\[
L(f)'(\tau-s)=\sum_{j\in \mathbb{Z}}b_j(\tau-s)e^{\tau|j|^2+ix\cdot j}\hat{g}_j.
\]
Consequently, we have
\[
e^{sA}L(f)'(\tau-s)=\sum_{j\in \mathbb{Z}}b_j(\tau-s)e^{(\tau-s)|j|^2+ix\cdot j}\hat{g}_j.
\]
Let
\[
T(g):=\sum_{j\in \mathbb{Z}}\int_0^\tau b_j(s)e^{s |j|^2+ix\cdot j} \hat{g}_j ds.
\]
With this new notation, \eqref{ie1} takes the form
\begin{equation}\label{ie2}
g+T(g)=k.
\end{equation}
As 
\[
|\widehat{T(g)}_j|\le \epsilon_j|\hat{g}_j|,
\]
we get $T\in \mathscr{B}(\mathfrak{R}_\tau(\mathbb{T}^n))$ and $\|T\|_{\mathscr{B}(\mathfrak{R}_\tau(\mathbb{T}^n))}\le \epsilon$. Therefore, $I+T$ is invertible and hence \eqref{ie2} admits a unique solution $g\in \mathfrak{R}_\tau(\mathbb{T}^n)$ given by $g=(I+T)^{-1}k$. Consequently, \eqref{ie1} possesses a unique solution $f=\mathfrak{I}_\tau^{-1}g\in B^0(\mathbb{T}^n)$. Furthermore, we have
\[
\|f\|_{B^0(\mathbb{T}^n)}=\|g\|_{\mathfrak{R}_\tau(\mathbb{T}^n)}\le (1-\epsilon)^{-1}\|k\|_{\mathfrak{R}_\tau(\mathbb{T}^n)}.
\]
The proof is complete.
\end{proof}

%
%
%
%
%
%
%
%
%

\bibliographystyle{amsplain}
\bibliography{refs}

\end{document}